\documentclass[12pt,a4paper,reqno]{amsart} 
\usepackage{footnote}
\usepackage{amssymb}
\usepackage{latexsym}
\usepackage{amsmath}
\usepackage{mathrsfs}
\usepackage[X2,T1]{fontenc}
\usepackage[applemac]{inputenc}
\usepackage{cite}
\usepackage{color}
\usepackage{calc}                   

\newcommand{\scal}[2]{\langle #1,#2\rangle}
\newcommand{\rr}[1]{\mathbf R^{#1}}

\newcommand{\cc}[1]{\mathbf C^{#1}}
\newcommand{\nm}[2]{\Vert #1\Vert _{#2}}

\newcommand{\sets}[2]{\{ \, #1\, ;\, #2\, \} }

\newcommand{\cdo}{\, \cdot \, }
\newcommand{\supp}{\operatorname{supp}}

\newcommand{\eabs}[1]{\langle #1\rangle}     

\newcommand{\vrum}{\vspace{0.1cm}}

\newcommand{\IM}{\operatorname{Im}}

\newcommand{\nn}[1]{{\mathbf N}^{#1}}
\newcommand{\maclA}{\mathcal A}

\newcommand{\maclH}{\mathcal H}
\newcommand{\maclE}{\mathcal E}

\newcommand{\maclR}{\mathcal R}
\newcommand{\maclS}{\mathcal S}

\newcommand{\mascE}{\mathscr E}

\newcommand{\mascH}{\mathscr H}

\newcommand{\mascS}{\mathscr S}

\numberwithin{equation}{section}          

\newtheorem{thm}{Theorem}
\numberwithin{thm}{section}

\newcommand{\rubrik}{}
\newtheorem{prop}[thm]{Proposition}

\newtheorem{lemma}[thm]{Lemma}

\theoremstyle{definition}

\newtheorem{defn}[thm]{Definition}

\theoremstyle{remark}
\newtheorem{rem}[thm]{Remark}

\author{Joachim Toft}

\address{Department of Mathematics,
Linn{\ae}us University, V{\"a}xj{\"o}, Sweden}

\email{joachim.toft@lnu.se}

\title{Paley-Wiener properties for spaces of power series expansions}

\keywords{Bargmann transform}

\subjclass[2010]{primary 46F05; 32A25; 32A36;
secondary 35Q40; 30Gxx}

\begin{document}


\begin{abstract}
We extend  Paley-Wiener results in the Bargmann setting deduced
in \cite{NaPfTo1} to larger class of power series expansions. At the
same time we deduce characterisations of all Pilipovi{\'c} spaces
and their distributions (and not only of low orders as in \cite{NaPfTo1}).
\end{abstract}

\maketitle

\par

\section{Introduction}\label{sec0}

\par

Classical Paley-Wiener theorems characterize functions and distributions
with certain restricted supports in terms of estimates of their 
Fourier-Laplace transforms. For example, let $f$ be a distribution
on $\rr d$ and let $B_{r_0}(0)\subseteq \rr d$ be the ball with center
at origin and radius $r_0$. Then $f$ is supported in $B_{r_0}(0)$ if 
and only if 
$$
|\widehat{f}(\zeta)| \lesssim \eabs\zeta ^{N} e^{r_0 |\IM (\zeta )|}, 
\quad \zeta\in \cc d,
$$
for some $N\ge 0$.

\par

In \cite{NaPfTo1,Toft15} a type of Paley-Wiener results were deduced for
certain spaces of entire functions, where the usual Fourier transform were
replaced by the reproducing kernel $\Pi _A$ of the Bargmann transform.
(See \cite{Ho1} and Section \ref{sec1} for notations.)
This reproducing kernel is an analytic global Fourier-Laplace transform
with respect to a suitable Gaussian measure. For example, in \cite{Toft15}
it is proved that if  $A(\cc d)$ is the set of all entire functions and
$\maclA _{\flat _\sigma}(\cc d)$ ($\maclA _{0,\flat _\sigma}(\cc d)$),
$\sigma >0$, is the set of all $F\in A(\cc d)$
such that $|F(z)|\lesssim e^{r|z|^{\frac {2\sigma}{\sigma +1}}}$ for some $r>0$
(for every $r>0$),
then
\begin{equation}\label{Eq:FirstRel}
\maclA _{\flat _1}(\cc d)
=
\Pi _A(\mascE '(\cc d)) = \Pi _A(\mascE '(\cc d)\cap L^\infty (\cc d)).
\end{equation}
The latter equality in \eqref{Eq:FirstRel} is in \cite{NaPfTo1}
refined into the relation
$$
F_0\in \maclA _{\flat _1}(\cc d)
\quad \Leftrightarrow \quad
F_0 = \Pi _A(\chi \cdot F)
$$
for some characteristic function $\chi$ of a polydisc $D$,
centered at origin, and a function $F$ which is defined and
analytic in a neighbourhood of $D$. The spaces
$\maclA _{0,\flat _1}(\cc d)$ is also characterized in
\cite{NaPfTo1} by
$$
\maclA _{0,\flat _1}(\cc d)
=
\Pi _A(\chi \cdot A(\cc d))
$$
when $\chi$ above is fixed.
%

\par

In fact, in \cite{NaPfTo1}, similar characterizations are deduced for
all $\maclA _{\flat _\sigma}(\cc d)$ and $\maclA _{0,\flat _\sigma}(\cc d)$
with $\sigma \le 1$ as well as
for the set $\maclA _s(\cc d)$ when $s\in [0,\frac 12)$ which consists
of all $F\in A(\cc d)$ such that $|F(z)|\lesssim e^{r(\log \eabs z)^{\frac 1{1-2s}}}$.
For example, let $\chi$ be the characteristic
function of a polydisc, centered at origin in $\cc d$. Then it is proved in
\cite{NaPfTo1} that
\begin{alignat*}{3}
\maclA _{\flat _\sigma}(\cc d)
&=
\sets {\Pi _A(F\cdot \chi )}{F\in \maclA _{\flat _{\sigma _0}}(\cc d) }&
\quad &\text{when}&\quad
\sigma &<\frac 12,\ \sigma _0=\frac \sigma {1-2\sigma}
\intertext{and}
\maclA _s(\cc d)
&=
\sets {\Pi _A(F\cdot \chi )}{F\in \maclA _s(\cc d)}&
\quad &\text{when}&\quad
s &<\frac 12.
\end{alignat*}

\par

In Section \ref{sec2} we extend the set of characterizations given
in \cite{NaPfTo1} in different ways. Firstly we show that we can choose
$\chi$ above in a larger class of functions and measures. Secondly,
we characterize strictly larger spaces than $\maclA _{\flat _1}(\cc d)$.
More specifically, we characterize all the spaces
\begin{alignat}{6}
&\maclA _{\flat _\sigma}(\cc d)&
\quad &\text{and}&\quad
&\maclA _s(\cc d) &
\qquad
\big (
&\maclA _{0,\flat _\sigma}(\cc d)&
\quad &\text{and}& \quad
&\maclA _{0,s}(\cc d)
\big ),
\label{Eq:TheSmallerSpaces}
\intertext{for any $s,\sigma \in \mathbf R$, as well as the larger
spaces}
&\maclA _{\flat _\sigma}'(\cc d)&
\quad &\text{and}& \quad
&\maclA _s'(\cc d)&
\qquad
\big (
&\maclA _{0,\flat _\sigma}'(\cc d)&
\quad &\text{and}& \quad
&\maclA _{0,s}'(\cc d)
\big ),
\label{Eq:TheLargerSpaces}
\end{alignat}
the sets of all formal power series expansions
\begin{equation}\label{Eq:PowerSeriesExpIntro}
F(z) = \sum _{\alpha \in \nn d}c(F,\alpha )e_\alpha (z),\qquad
e_\alpha (z)= \frac {z^\alpha}{\sqrt {\alpha !}},\ \alpha \in \nn d
\end{equation}
such that
$$
|c(F,\alpha )|\lesssim h^{|\alpha |}\alpha !^{\frac 1{2\sigma}}
\quad \text{respective}\quad
|c(F,\alpha )|\lesssim e^{r|\alpha |^{\frac 1{2s}}}
$$
for every $h,r>0$ (for some $h,r>0$).
We notice that the spaces \eqref{Eq:TheSmallerSpaces} discussed
above can also be described as the sets of all formal power series
expansions \eqref{Eq:PowerSeriesExpIntro}
such that
$$
|c(F,\alpha )|\lesssim h^{|\alpha |}\alpha !^{-\frac 1{2\sigma}}
\quad \text{respective}\quad
|c(F,\alpha )| \lesssim e^{-r|\alpha |^{\frac 1{2s}}}
$$
for some $h,r>0$ (for every $h,r>0$) (cf.
\cite{FeGaTo2,Toft15} and Section \ref{sec1}).
We remark that for $s\ge \frac 12$, $\maclA _s(\cc d)$ and $\maclA _{0,s}(\cc d)$,
defined in this way, are still spaces of entire functions,
but with other types of estimates, compared to the case $s<\frac 12$
considered above (cf. \cite{Toft15} and Section \ref{sec1}).

\par

By the definitions we have
$\maclA _{\flat _1}'(\cc d)=A(\cc d)$, and by using the convention
$$
s<\flat _{\sigma _1} <\flat _{\sigma _2} <\frac 12
\quad \text{when}\quad \sigma _1<\sigma _2,
$$
it follows that
\begin{multline*}
\maclA _0(\cc d)\subseteq \maclA _{0,s_1}(\cc d) \subseteq \maclA _{s_1}(\cc d)
\subseteq
\maclA _{0,s_2}(\cc d)
\\[1ex]
\subseteq \maclA _{0,s_2}'(\cc d) \subseteq \maclA _{s_1}'(\cc d)
\subseteq \maclA _{0,s_1}'(\cc d),
\\[1ex]
\text{when}
\quad
s_1,s_2 \in
\mathbf R\cup \{ \flat _\sigma\} _{\sigma >0},\ s_1<s_2
\end{multline*}
(cf. \cite{Toft15}). In particular, if $\sigma >1$, then
$\maclA _{\flat _\sigma}'(\cc d)$
and $\maclA _{0,\flat _\sigma}'(\cc d)$ are contained in $A(\cc d)$,
and in \cite{Toft15}, it is proved that the former spaces consists of
all $F\in A(\cc d)$ such that
$|F(z)|\lesssim e^{r|z|^{\frac {2\sigma}{\sigma -1}}}$
for every $r>0$ respective for some $r>0$.

\par

In Section \ref{sec2} we show among others that if $T_{\chi}$ is the map
$F\mapsto \Pi _A(\chi \cdot F)$ with $\chi$ as above, then
the mappings
\begin{alignat*}{5}
T_{\chi}\, &:\, & \maclA _{0,s}'(\cc d) &\to  \maclA _{0,s}'(\cc d), & \quad
T_{\chi}\, &:\, & \maclA _{s}'(\cc d) &\to  \maclA _{s}'(\cc d), & \quad s &<\frac 12,
\\[1ex]
T_{\chi}\, &:\, & \maclA _{0,\flat _{\sigma _0}}'(\cc d) &\to 
\maclA _{0,\flat _\sigma}'(\cc d), & \quad
T_{\chi}\, &:\, & \maclA _{\flat _{\sigma _0}}'(\cc d) &\to  
\maclA _{\flat _{\sigma}}'(\cc d),& \quad \sigma _0 &= \frac \sigma{2\sigma +1},
\intertext{when $\sigma >0$,}
T_{\chi}\, &:\, & \maclA _{0,\flat _{\sigma _0}}'(\cc d) &\to 
\maclA _{\flat _\sigma}(\cc d), & \quad
T_{\chi}\, &:\, & \maclA _{\flat _{\sigma _0}}'(\cc d) &\to  
\maclA _{0,\flat _{\sigma}}(\cc d),& \quad \sigma _0 &= \frac \sigma{2\sigma -1},
\intertext{when $\sigma >\frac 12$, and}
T_{\chi}\, &:\, & \maclA _{0,\flat _{\sigma _0}}(\cc d) &\to 
\maclA _{0,\flat _\sigma}(\cc d), & \quad
T_{\chi}\, &:\, & \maclA _{\flat _{\sigma _0}}(\cc d) &\to  
\maclA _{\flat _{\sigma}}(\cc d),& \quad \sigma _0 &= \frac \sigma{1-2\sigma},
\end{alignat*}
when $\sigma <\frac 12$, are well-defined and bijective.

\par

In Section \ref{sec3} we apply these mapping properties to deduce
characterizations of Pilipovi{\'c} spaces and their distribution spaces
in terms of images of the adjoint of Gaussian windowed short-time
Fourier transform, acting on suitable spaces which are strongly linked to
the spaces in \eqref{Eq:TheSmallerSpaces} and \eqref{Eq:TheLargerSpaces}.
These spaces are
obtained by imposing the same type of estimates on their Hermite
coefficients as for the power series coefficients of the spaces in
\eqref{Eq:TheSmallerSpaces} and \eqref{Eq:TheLargerSpaces}. It turns
out that Pilipovi{\'c} spaces and their distribution spaces are exactly
the counter images of the spaces in
\eqref{Eq:TheSmallerSpaces} and \eqref{Eq:TheLargerSpaces}
under the Bargmann transform (cf. \cite{Toft15}).

\par

\section{Preliminaries}\label{sec1}

\par

In this section we recall some basic facts. We start by discussing
Pilipovi{\'c} spaces and some of their properties.
Then we recall some facts on modulation spaces. Finally
we recall the Bargmann transform and some of its
mapping properties, and introduce suitable classes of entire functions on
$\cc d$.

\par

\subsection{Spaces of sequences}\label{subsec1.1}

\par

The definitions of Pilipovi{\'c} spaces and spaces of power series
expansions are based on certain spaces of sequences on $\nn d$, indexed
by the extended set
$$
\mathbf R_\flat = \mathbf R_+\bigcup \sets {\flat _\sigma}{\sigma \in \mathbf R_+},
$$
of $\mathbf R_+$. We extend the extend the inequality relations on
$\mathbf R_+$ to the set $\mathbf R_\flat$, by letting
$$
s_1<\flat _\sigma <s_2
\quad \text{and}\quad \flat _{\sigma _1}<\flat _{\sigma _2}
$$
when $s_1<\frac 12\le s_2$ and $\sigma _1<\sigma _2$. (Cf. \cite{Toft15}.)

\par

\begin{defn}\label{Def:SeqSpaces}
Let $s\in \mathbf R_\flat$ and $\sigma \in \mathbf R_+$.
\begin{enumerate}
\item The set $\ell _0'(\nn d)$ consists of all formal
sequences $a=\{ a(\alpha )\}_{\alpha \in \nn d}\subseteq \mathbf C$,
and $\ell _0(\nn d)$ is the set of all $a\in \ell _0'(\nn d)$ such that
$a(\alpha )\neq 0$ for at most finite numbers of $\alpha \in \nn d$;

\vrum

\item If $r,s\in \mathbf R_+$
and $a\in \ell _0'(\nn d)$, then the Banach spaces
$$
\ell _{r,s}^\infty (\nn d),
\quad
\ell _{r,s}^{\infty ,*} (\nn d),
\quad
\ell _{r,\flat _\sigma}^\infty (\nn d)
\quad \text{and}\quad
\ell _{r,\flat _\sigma }^{\infty ,*} (\nn d)
$$
consists of all $a\in \ell _0'(\nn d)$ such that their corresponding
norms
\begin{alignat*}{2}
\nm a{\ell _{r,s}^\infty}
&\equiv
\sup _{\alpha \in \nn d} |a(\alpha )e^{r|\alpha |^{\frac 1{2s}}}|,
&\quad
\nm a{\ell _{r,s}^{\infty ,*}}
&\equiv
\sup _{\alpha \in \nn d} |a(\alpha )e^{-r|\alpha |^{\frac 1{2s}}}|
\\[1ex]
\nm a{\ell _{r,\flat _\sigma}^\infty}
&\equiv
\sup _{\alpha \in \nn d} |a(\alpha )
r^{-|\alpha |} \alpha !^{\frac 1{2\sigma }}|,
&\quad
\nm a{\ell _{r,\flat _\sigma}^{\infty ,*}}
&\equiv
\sup _{\alpha \in \nn d} |a(\alpha )
r^{-|\alpha |} \alpha !^{-\frac 1{2\sigma }}|,
\end{alignat*}
respectively, are finite;

\vrum

\item The space $\ell _s(\nn d)$ ($\ell _{0,s}(\nn d)$) is the
inductive limit (projective limit) of $\ell _{r,s}^\infty (\nn d)$
with respect to $r>0$, and $\ell _s'(\nn d)$ ($\ell _{0,s}'(\nn d)$) is the
projective limit (inductive limit) of $\ell _{r,s}^{\infty ,*} (\nn d)$
with respect to $r>0$.
\end{enumerate}
\end{defn}

\par

We also let $\ell _{0,N}(\nn d)$ be the set of all $a\in \ell _0'(\nn d)$
such that $a(\alpha )=0$ when $|\alpha |\ge N$. Then
$$
\ell _0(\nn d)= \bigcup _{N\ge 0}\ell _{0,N}(\nn d),
$$
and $\ell _{0,N}(\nn d)$ is a Banach space under the norm
$$
\nm a{\ell _{0,N}} \equiv \sup _{|\alpha |\le N}|a(\alpha )|.
$$
We equip $\ell _0(\nn d)$ with the inductive limit topology of
$\ell _{0,N}(\nn d)$, and supply $\ell _0'(\nn d)$ with Fr{\'e}chet
space topology through the semi-norms $\nm \cdo{\ell _{0,N}}$.

\par

In what follows, $(\cdo ,\cdo )_{\mascH}$ denotes the scalar
product in the Hilbert space $\mascH$.

\par

\begin{rem}\label{Rem:SeqSpaces}
For the spaces in Definition \ref{Def:SeqSpaces}, the following is true.
We leave the verifications for the reader.
\begin{enumerate}
\item If $s\in \mathbf R_\flat$ then
\begin{alignat*}{2}
\ell _{0,s} (\nn d)&= \bigcap _{r>0} \ell ^\infty _{r,s}(\nn d),&
\quad
\ell _s(\nn d) &= \bigcup _{r>0} \ell ^\infty _{r,s}(\nn d),
\\[1ex]
\ell _s'(\nn d) &= \bigcap _{r>0} \ell ^{\infty ,*} _{r,s}(\nn d)&
\quad \text{and}\quad
\ell _{0,s}'(\nn d) &= \bigcup _{r>0} \ell ^{\infty ,*}_{r,s}(\nn d) \text ;
\end{alignat*}

\vrum

\item The space $\ell _0(\nn d)$ is dense in all of the spaces
in Definition \ref{Def:SeqSpaces} (1) and (3);

\vrum

\item If $s_1\in \mathbf R_\flat$ and
$s_2\in \overline{\mathbf R}_\flat$, then the map
$(a,b)\mapsto (a,b)_{\ell ^2(\nn d)}$ from $\ell _0(\nn d)\times \ell _0(\nn d)$
to $\mathbf C$ is uniquely extendable to continuous mappings from
\begin{alignat*}{2}
&\ell _{0,s_1}'(\nn d)\times \ell _{0,s_1}(\nn d)&
\quad
&\ell _{0,s_1}(\nn d)\times \ell _{0,s_1}'(\nn d),
\\[1ex]
&\ell _{s_2}'(\nn d)\times \ell _{s_2}(\nn d)&
\quad \text{or}\quad
&\ell _{s_2}(\nn d)\times \ell _{s_2}'(\nn d)
\end{alignat*}
to $\mathbf C$. The duals of $\ell _{0,s_1}(\nn d)$ and $\ell _{s_2}(\nn d)$
can be identified by $\ell _{0,s_1}(\nn d)$ respective $\ell _{s_2}(\nn d)$,
through these extensions of the form $(\cdo ,\cdo )_{\ell ^2(\nn d)}$.
\end{enumerate}
\end{rem}

\par

\subsection{Pilipovi{\'c} spaces and spaces of power series expansions
on $\cc d$}\label{subsec1.2}

\par

We recall that
the Hermite function of order $\alpha \in \nn d$ is defined by
$$
h_\alpha (x) = \pi ^{-\frac d4}(-1)^{|\alpha |}
(2^{|\alpha |}\alpha !)^{-\frac 12}e^{\frac {|x|^2}2}
(\partial ^\alpha e^{-|x|^2}).
$$
It follows that
$$
h_{\alpha}(x)=   ( (2\pi )^{\frac d2} \alpha ! )^{-1}
e^{-\frac {|x|^2}2}p_{\alpha}(x),
$$
for some polynomial $p_\alpha$ of order $\alpha$
on $\rr d$,
called the Hermite polynomial of order $\alpha$. 
The Hermite functions are eigenfunctions to the Fourier transform, and
to the Harmonic oscillator $H_d\equiv |x|^2-\Delta$ which acts on functions
and (ultra-)distributions defined
on $\rr d$. More precisely, we have
$$
H_dh_\alpha = (2|\alpha |+d)h_\alpha ,\qquad H_d\equiv |x|^2-\Delta .
$$

\par

It is well-known that
the set of Hermite functions is a basis for $\mascS (\rr d)$ 
and an orthonormal basis for $L^2(\rr d)$ (cf. \cite{ReSi}).
In particular, if $f,g\in L^2(\rr d)$, then
$$
\nm f{L^2(\rr d)}^2 = \sum _{\alpha \in \nn d}|c_h(f,\alpha )|^2
\quad \text{and}\quad
(f,g)_{L^2(\rr d)} = \sum _{\alpha \in \nn d}c_h(f,\alpha )
\overline{c_h(g,\alpha )},
$$
where
\begin{align}
f(x) &= \sum _{\alpha \in \nn d}c_h(f,\alpha )h_\alpha (x)
\label{Eq:HermiteExp}
\intertext{is the Hermite seriers expansion of $f$, and}
c_h(f,\alpha ) &= (f,h_\alpha )_{L^2(\rr d)}
\label{Eq:HermiteCoeff}
\end{align}
is the Hermite coefficient of $f$ of order $\alpha \in \rr d$.

\par

We shall also consider formal power series expansions on $\cc d$, centered
at origin. That is, we shall consider formal expressions of the form
\begin{equation}\label{Eq:PowerSeriesExp}
F(z) = \sum _{\alpha \in \nn d}c(F,\alpha )e_\alpha (z),\qquad
e_\alpha (z)= \frac {z^\alpha}{\sqrt {\alpha !}},\ \alpha \in \nn d.
\end{equation}

\par

\begin{defn}\label{Def:PilPowerSpaces}
The set $\maclH _0'(\rr d)$ consists of all formal hermite series
expansions \eqref{Eq:HermiteExp}, and $\maclA _0'(\cc d)$
consists of all formal power series expansions \eqref{Eq:PowerSeriesExp}.
The set $\maclH _0(\rr d)$ ($\maclA _0(\cc d)$) consists
of all $f\in \maclH _0'(\rr d)$ ($F\in \maclA _0'(\cc d)$)
such that $c_h(f,\alpha )\neq 0$ ($c(F,\alpha )\neq 0$) 
for at most finite numbers of $\alpha \in \nn d$.
\begin{enumerate}
\item If $s_1\in \mathbf R_\flat$ and $s_2\in \overline{\mathbf R}_\flat$,
then
\begin{align}
\maclH _{s_2}(\rr d), \quad \maclH _{0,s_1}(\rr d),
\quad \maclH _{0,s_1}'(\rr d)
\quad \text{and}\quad
\maclH _{s_2}'(\rr d)
\label{Eq:PilSpaces}
\end{align}
are the sets of all Hermite series expansions \eqref{Eq:HermiteExp} such that
their coefficients $\{ c_h(f,\alpha ) \} _{\alpha \in \nn d}$ belong
to $\ell _{s_2}(\nn d)$, $\ell _{0,s_1}(\nn d)$, $\ell _{0,s_1}'(\nn d)$
respective $\ell _{s_2}'(\nn d)$;

\vrum

\item If $s_1\in \mathbf R_\flat$ and $s_2\in \overline{\mathbf R}_\flat$,
then
\begin{align}
\maclA _{s_2}(\cc d), \quad \maclA _{0,s_1}(\cc d),
\quad \maclA _{0,s_1}'(\cc d)
\quad \text{and}\quad
\maclA _{s_2}'(\cc d)
\label{Eq:PowerSpaces}
\end{align}
are the sets of all power series expansions \eqref{Eq:PowerSeriesExp} such that
their coefficients $\{ c(F,\alpha ) \} _{\alpha \in \nn d}$ belong
to $\ell _{s_2}(\nn d)$, $\ell _{0,s_1}(\nn d)$, $\ell _{0,s_1}'(\nn d)$
respective $\ell _{s_2}'(\nn d)$.
\end{enumerate}
\end{defn}

\par

The spaces $\maclH _s(\rr d)$ and $\maclH _{0,s}(\rr d)$ in Definition
\ref{Def:PilPowerSpaces} are called
\emph{Pilipovi{\'c} spaces} of \emph{Roumieu} respectively \emph{Beurling types} of order $s$,
and
$\maclH _s'(\rr d)$ and $\maclH _{0,s}'(\rr d)$ are called
\emph{Pilipovi{\'c} distribution spaces} of \emph{Roumieu} respectively
\emph{Beurling types} of order $s$.

\par

\begin{rem}\label{Rem:PilPowerSpaces}
Let  $T_{\maclH}$ be the map from $\ell _0'(\nn d)$ to $\maclH _0'(\rr d)$
which takes the sequence $\{ c_h(f,\alpha ) \} _{\alpha \in \nn d}$
to the expansion \eqref{Eq:HermiteExp}, and let $T_{\maclA}$ be the
map from $\ell _0'(\nn d)$ to $\maclA _0'(\cc d)$
which takes the sequence $\{ c(F,\alpha ) \} _{\alpha \in \nn d}$
to the expansion \eqref{Eq:PowerSeriesExp}. Then it is clear
that $T_{\maclH}$ restricts to bijective mappings from
\begin{align}
\ell _{s_2}(\nn d), \quad \ell _{0,s_1}(\nn d),
\quad \ell _{0,s_1}'(\nn d)
\quad \text{and}\quad
\ell _{s_2}'(\nn d)
\label{Eq:SeqSpaces}
\end{align}
to respective spaces in \eqref{Eq:PilSpaces}, and that
$T_{\maclA}$ restricts to bijective mappings from the spaces
in \eqref{Eq:SeqSpaces} to respecive spaces in
\eqref{Eq:PowerSpaces}.

\par

We let the topologies of the spaces in \eqref{Eq:PilSpaces}
and \eqref{Eq:PowerSpaces} be inherited from the topologies
of respective spaces in \eqref{Eq:SeqSpaces}, through the
mappings $T_{\maclH}$ and $T_{\maclA}$.
\end{rem}

\par

The following result shows that Pilipovi{\'c} spaces of order
$s\in \mathbf R_+$ may in
convenient ways be characterized by estimates of powers
of harmonic oscillators applied on the involved functions.
We omit the proof since the result follows in the case $s\ge \frac 12$
from \cite{Pil2} and from \cite{Toft15} for general $s$.

\par

\begin{prop}\label{Prop:PilSpaceChar}
Let $s\ge 0$ ($s> 0$) be real, $H_d=|x|^2-\Delta$ be the
harmonic oscillator on $\rr d$ and set
$$
\nm f{(r,s)}\equiv \sup _{N\in \mathbf N}
\left (
\frac {\nm {H_d^Nf}{L^\infty (\rr d)}}{r^NN!^{2s}}
\right ),
\quad f\in C^\infty (\rr d).
$$
Then $f\in \maclH _s(\rr d)$ ($f\in \maclH _{0,s}(\rr d)$),
if and only if $\nm f{(r,s)}<\infty$ for some $r>0$ (for every
$r>0$). The topologies of $\maclH _s(\rr d)$ and
$\maclH _{0,s}(\rr d)$ agree with the inductive and projective
limit topologies, respectively, induced by the semi-norms
$\nm \cdo{(r,s)}$, $r>0$.
\end{prop}

\par

\begin{rem}\label{Remark:GSHermite}
Let $\maclS _s(\rr d)$ and $\Sigma _s(\rr d)$
be the Fourier invariant Gelfand-Shilov spaces of order $s\in \mathbf R_+$
and of Rourmeu and Beurling types respectively (see \cite{Toft15} for notations).
Then it is proved in \cite{Pil1,Pil2} that
\begin{alignat*}{2}
\maclH _{0,s}(\rr d) &=\Sigma _s(\rr d)\neq \{ 0\} ,& \quad s&> \frac 12,
\\[1ex]
\maclH _{0,s}(\rr d) &\neq\Sigma _s(\rr d) = \{ 0\} ,& \ s&= \frac 12,
\intertext{and}
\maclH _s(\rr d) &=\maclS _s(\rr d)\neq \{ 0\} ,& \quad s&\ge \frac 12
\end{alignat*}
\end{rem}

\par


\par

\subsection{Spaces of entire functions and
the Bargmann transform}\label{subsec1.4}
Let $\Omega \subseteq \cc d$ be open and let $\Omega _0\subseteq \cc d$
be non-empty (but not necessary open). Then $A(\Omega )$ is the set of all
analytic functions in $\Omega$, and
$$
A(\Omega _0) = \bigcup A(\Omega ),
$$
where the union is taken over all open sets $\Omega \subseteq \cc d$
such that $\Omega _0\subseteq \Omega$. In particular, if $z_0\in \cc d$
is fixed, then $A(\{z_0\} )$ is the set of all complex-valued functions
which are defined and analytic near $z_0$.

\par

We shall now consider the Bargmann transform. We set
\begin{gather*}
\scal zw = \sum _{j=1}^dz_jw_j
\quad \text{and}\quad (z,w) = \scal z{\overline w}
,\quad \text{when}
\\[1ex]
z=(z_1,\dots ,z_d) \in \cc d
\quad \text{and} \quad
w=(w_1,\dots ,w_d)\in \cc d,
\end{gather*}
and otherwise $\scal \cdo \cdo $ denotes the duality between test
function spaces and their corresponding duals.
The Bargmann transform $\mathfrak V_df$ of $f\in L^2(\rr d)$
is defined by the formula
\begin{equation}\label{Eq:BargmTransf}
(\mathfrak V_df)(z) =\pi ^{-\frac d4}\int _{\rr d}\exp \Big ( -\frac 12(\scal
z z+|y|^2)+2^{\frac 12}\scal zy \Big )f(y)\, dy
\end{equation}
(cf. \cite{B1}). We note that if $f\in
L^2(\rr d)$, then the Bargmann transform
$\mathfrak V_df$ of $f$ is the entire function on $\cc d$, given by
$$
(\mathfrak V_df)(z) =\int _{\rr d}\mathfrak A_d(z,y)f(y)\, dy,
$$
or
\begin{equation}\label{bargdistrform}
(\mathfrak V_df)(z) =\scal f{\mathfrak A_d(z,\cdo )},
\end{equation}
where the Bargmann kernel $\mathfrak A_d$ is given by
$$
\mathfrak A_d(z,y)=\pi ^{-\frac d4} \exp \Big ( -\frac 12(\scal
zz+|y|^2)+2^{\frac 12}\scal zy\Big ).
$$
Evidently, the right-hand side in \eqref{bargdistrform} makes sense
when $f\in \maclS _{\frac 12}'(\rr d)$ and defines an element in $A(\cc d)$,
since $y\mapsto \mathfrak A_d(z,y)$ can be interpreted as an element
in $\maclS _{\frac 12} (\rr d)$ with values in $A(\cc d)$.

\par

It was proved in \cite{B1} that $f\mapsto \mathfrak V_df$ is a bijective
and isometric map  from $L^2(\rr d)$ to the Hilbert space $A^2(\cc d)
\equiv B^2(\cc d)\cap A(\cc d)$, where $B^2(\cc d)$ consists of all
measurable functions $F$ on $\cc  d$ such that
\begin{equation}\label{A2norm}
\nm F{B^2}\equiv \Big ( \int _{\cc d}|F(z)|^2d\mu (z)  \Big )^{\frac 12}<\infty .
\end{equation}
Here $d\mu (z)=\pi ^{-d} e^{-|z|^2}\, d\lambda (z)$, where $d\lambda (z)$ is the
Lebesgue measure on $\cc d$. We recall that $A^2(\cc d)$ and $B^2(\cc d)$
are Hilbert spaces, where the scalar product are given by
\begin{equation}\label{A2scalar}
(F,G)_{B^2}\equiv  \int _{\cc d} F(z)\overline {G(z)}\, d\mu (z),
\quad F,G\in B^2(\cc d).
\end{equation}
If $F,G\in A^2(\cc d)$, then we set $\nm F{A^2}=\nm F{B^2}$
and $(F,G)_{A^2}=(F,G)_{B^2}$.

\par

Furthermore, Bargmann showed that there is a convenient reproducing
formula on $A^2(\cc d)$. More precisely, let
\begin{equation}\label{reproducing}
(\Pi _AF)(z) \equiv \int _{\cc d}F(w)e^{(z,w)}\, d\mu (w),
\end{equation}
when $z\mapsto F(z)e^{R|z|-|z|^2}$ belongs to $L^1(\cc d)$
for every $R\ge 0$. Then it is proved in \cite{B1,B2} that
$\Pi _A$ is the orthogonal
projection of $B^2(\cc d)$ onto $A^2(\cc d)$. In particular,
$\Pi _AF =F$ when $F\in A^2(\cc d)$.

\medspace

In \cite{B1} it is also proved that
\begin{equation}\label{BargmannHermite}
\mathfrak V_dh_\alpha  = e_\alpha ,\quad \text{where}\quad
e_\alpha (z)\equiv \frac {z^\alpha}{\sqrt {\alpha !}},\quad z\in \cc d .
\end{equation}
In particular, the Bargmann transform maps the orthonormal basis
$\{ h_\alpha \}_{\alpha \in \nn d}$ in $L^2(\rr d)$ bijectively into the
orthonormal basis $\{ e_\alpha \}_{\alpha \in \nn d}$ of monomials
in $A^2(\cc d)$.

\par

For general $f\in \maclH _0'(\rr d)$ we now set
\begin{equation}\label{Eq:GeneralBargmannDef}
\mathfrak V_df \equiv (T_{\maclA}\circ T_{\maclH}^{-1})f,
\qquad f\in \maclH _0'(\rr d),
\end{equation}
where $T_{\maclH}$ and $T_{\maclA}$ are given by Remark
\ref{Rem:PilPowerSpaces}. It follows from \eqref{BargmannHermite}
that $\mathfrak V_df$ in \eqref{Eq:GeneralBargmannDef} agrees with
$\mathfrak V_df$ in \eqref{Eq:BargmTransf} when $f\in L^2(\rr d)$,
and that this is the only way to extend the Bargmann transform to the
space $\maclH _0'(\rr d)$. It follows that
$\mathfrak V_d=T_{\maclA}\circ T_{\maclH}^{-1}$
is a homeomorphism from $\maclH _0'(\rr d)$
to $\maclA _0'(\cc d)$, which restricts to homeomorphisms
from the spaces in \eqref{Eq:PilSpaces} to the spaces
in \eqref{Eq:PowerSpaces}, respectively.
If $f\in \maclH _0'(\rr d)$ and $F\in \maclA _0'(\cc d)$
are given by \eqref{Eq:HermiteExp} and \eqref{Eq:PowerSeriesExp}
with $c_h(f,\alpha )= c(F,\alpha )$ for all $\alpha \in \nn d$,
then it follows that $\mathfrak V_df =F$.

\par

It follows that if $f,g\in L^2(\rr d)$ and $F,G\in A^2(\cc d)$, then
\begin{equation}\label{Scalarproducts}
\begin{aligned}
(f,g)_{L^2(\rr d)} &= \sum _{\alpha \in \nn d}c_h(f,\alpha ) \overline {c_h(g,\alpha )},
\\[1ex]
(F,G)_{A^2(\cc d)} &= \sum _{\alpha \in \nn d}c(F,\alpha ) \overline {c(G,\alpha )}.
\end{aligned}
\end{equation}

By the definitions we get the following proposition on duality for Pilipovi{\'c}
spaces and their Bargmann images. The details are left
for the reader.

\par


\begin{prop}
Let $s_1\in \mathbf R_\flat$ and $s_2\in \overline{\mathbf R}_\flat$. Then 
the form $(\cdo ,\cdo )_{L^2(\rr d)}$ on
$\maclH _0(\rr d)\times \maclH _0(\rr d)$ is uniquely extendable to
sesqui-linear forms on
\begin{alignat*}{2}
&\maclH _{s_2}'(\rr d)\times \maclH _{s_2}(\rr d),
&\quad
&\maclH _{s_2}(\rr d)\times \maclH _{s_2}'(\rr d),
\\[1ex]
&\maclH _{0,s_1}'(\rr d)\times \maclH _{0,s_1}(\rr d)
&\quad \text{and on}\quad
&\maclH _{0,s_1}(\rr d)\times \maclH _{0,s_1}'(\rr d).
\end{alignat*}
The duals of $\maclH _{s_2}(\rr d)$ and $\maclH _{0,s_1}(\rr d)$
are equal to $\maclH _{s_2}'(\rr d)$ and $\maclH _{0,s_1}'(\rr d)$,
respectively, through the form $(\cdo ,\cdo )_{L^2(\rr d)}$.

\par

The same holds true if the spaces in \eqref{Eq:PilSpaces}
and the form $(\cdo ,\cdo )_{L^2(\rr d)}$ are replaced by
corresponding spaces in \eqref{Eq:PowerSpaces} and the form
$(\cdo ,\cdo )_{A^2(\cc d)}$, at each occurrence.
\end{prop}

\par

If $s\in \overline {\mathbf R_\flat}$, $f\in \maclH _s(\rr d)$, $g\in \maclH _s'(\rr d)$,
$F \in \maclA _s(\cc d)$ and $G \in \maclA _s'(\cc d)$, then $(f,g)_{L^2(\rr d)}$
and $(F,G)_{A^2(\cc d)}$ are defined by the formula \eqref{Scalarproducts}.
It follows that
\begin{equation}\label{ScalarproductsRel}
c_h(f,\alpha ) = c(F,\alpha )
\quad \text{when}\quad
F=\mathfrak V_df ,\ G=\mathfrak V_dg .
\end{equation}
holds for such choices of $f$ and $g$.
Furthermore, the duals of $\maclH _s(\rr d)$
and $\maclA _s(\cc d)$ can be identified with $\maclH _s'(\rr d)$
and $\maclA _s'(\cc d)$, respectively, through the forms in \eqref{Scalarproducts}.
The same holds true with
\begin{alignat*}{5}
&\maclH _{0,s}, &\quad &\maclH _{0,s}' ,& \quad &\maclA _{0,s}, &\quad &\text{and}&\quad
&\maclA _{0,s}'
\intertext{in place of}
&\maclH _s, &\quad &\maclH _s' , &\quad &\maclA _s, &\quad &\text{and}&\quad
&\maclA _s',
\end{alignat*}
respectively, at each occurrence.

\par

\begin{rem}\label{Rem:AnalSpacesIdent}
In \cite{Toft15}, the spaces in \eqref{Eq:PowerSpaces}, contained in
$\maclA _{0,\flat _1}'(\cc d)$ are identified as follows as spaces
of analytic functions:
\begin{enumerate}
\item if $s\in [0,\frac 12 )$ and $\sigma >0$, then $\maclA _s(\cc d)$
($\maclA _{0,s}(\cc d)$) is equal to
\begin{align*}
&\sets {F\in A(\cc d)}{|F(z)|\lesssim e^{r(\log \eabs z)^{\frac 1{1-2s}}}
\ \text{for some (every) $r>0$}}
\end{align*}

\vrum

\item if $\sigma _1>0$ and $\sigma _2>1$, then
$\maclA _{\flat _{\sigma _1}}(\cc d)$
($\maclA _{0,\flat _{\sigma _1}}(\cc d)$) is equal to
\begin{align*}
&\sets {F\in A(\cc d)}{|F(z)|\lesssim e^{r|z|^{\frac {2\sigma _1}{\sigma _1 +1}} }
\ \text{for some (every) $r>0$}},
\intertext{and $\maclA _{\flat _{\sigma _2}}'(\cc d)$
($\maclA _{0,\flat _{\sigma _2}}'(\cc d)$) is equal to}
&\sets {F\in A(\cc d)}{|F(z)|\lesssim e^{r|z|^{\frac {2\sigma_2}{\sigma _2 -1}} }
\ \text{for every (some) $r>0$}}.
\end{align*}
Furthermore,  $\maclA _{\flat _1}'(\cc d) = A(\cc d)$,
and $\maclA _{0,\flat _1}'(\cc d) = A(\{ 0 \})$;

\vrum

\item if $s_0=\frac 12$, then $\maclA _{0,s_0}(\cc d)$ ($\maclA _{0,s_0}'(\cc d)$)
is equal to
\begin{align*}
\sets {F\in A(\cc d)}{|F(z)|\lesssim e^{r|z|^2}
\ \text{for every (some) $r>0$}};
\end{align*}

\vrum

\item if $s\ge \frac 12$ ($s> \frac 12$), then $\maclA _{s}(\cc d)$
($\maclA _{0,s}(\cc d)$)
is equal to
\begin{align*}
&\sets {F\in A(\cc d)}{|F(z)|\lesssim e^{\frac 12\cdot |z|^2 -r |z|^{\frac 1{s}}}
\ \text{for some (every) $r>0$} }
\intertext{and $\maclA _{s}'(\cc d)$ ($\maclA _{0,s}'(\cc d)$) is equal to}
&\sets {F\in A(\cc d)}{|F(z)|\lesssim e^{\frac 12\cdot |z|^2 +r |z|^{\frac 1{s}}}
\ \text{for every (some) $r>0$} } .
\end{align*}
\end{enumerate}
\end{rem}

\par

Additionally to Remark \ref{Rem:AnalSpacesIdent} we have
$$
A(\cc d) = \bigcap _{r\in \rr d_+} A(D_{d,r}(z))
\quad \text{and}\quad
A(\{ 0 \}) = \bigcup _{r\in \rr d_+} A(D_{d,r}(z)).
$$
Here and in what follows,
$D_{d,r}(z_0)$ is the (open) polydisc
$$
\sets {z=(z_1,\dots ,z_d)\in \cc d}{|z_j-z_{0,j}|< r_j,\ j=1,\dots ,d},
$$
with center and radii given by
$$
z_0 =(z_{0,1},\dots ,z_{0,d}) \in \cc d
\quad \text{and}\quad
r=(r_1,\dots ,r_d)\in \mathbf [0,\infty )^d.
$$

\par

\subsection{A test function space introduced by Gr{\"o}chenig}

\par

In this section we recall some comparison results deduced in \cite{Toft15},
between a test function space, $\maclS _C(\rr d)$, introduced by Gr{\"o}chenig
in \cite{Gc2} to handle modulation spaces with elements in
spaces of ultra-distributions.

\par

The definition of $\maclS _C(\rr d)$ is given as follows.

\par

\begin{defn}\label{Def:SCSG}
The space
$\maclS _C(\rr d)$ consists of all
$f\in \mascS '(\rr d)$ such that
\begin{equation}\label{Eq:fSTFTAdjDef}
f(x) = \iint _{\rr {2d}} F(y,\eta )
e^{-(\frac 12|x-y|^2+|y|^2+|\eta |^2)}
e^{i(\frac 12\scal y\eta -\scal x\eta )}\, dyd\eta ,
\end{equation}
for some $F\in L^\infty (\rr {2d})\cap \mascE '(\rr {2d})$.
\end{defn}
%
%
%
%

\par

Evidently, we could have included the factors $e^{-(|y|^2+|\eta |^2)}$ and
$e^{\frac i2\scal y\eta}$ in the function $F(y,\eta )$ in \eqref{Eq:fSTFTAdjDef}.
The following reformulation of \cite[Lemma 4.9]{Toft15} justifies the separation. 
The result is essential when deducing the characterizations of Pilipovi{\'c}
spaces in Section \ref{sec3}.

\par

\begin{lemma}\label{GrochSpaceBargm}
Let $F\in L^\infty (\cc d)\cup \mascE '(\cc d)$. Then the Bargmann transform
of $f$ in \eqref{Eq:fSTFTAdjDef}
is given by $\Pi _AF_0$, where
\begin{equation}\label{Eq:FtoF0}
F_0(x+i\xi ) = (8\pi ^5)^{\frac d4}F(\sqrt 2 x,\sqrt 2 \xi ).
\end{equation}
\end{lemma}

\par



The first part of the next result follows from \cite[Theorem 4.10]{Toft15}
and the last part of the result follows \cite[Theorem 2.2]{NaPfTo1}. The
proof is therefore omitted.

\par

\begin{prop}\label{Prop:GrochPilipIdentity}
The following is true:
\begin{enumerate}
\item $\maclS _C(\rr d) = \maclH _{\flat _1}(\rr d)$;

\vrum

\item the image of $L^\infty (\cc d)\bigcap \mascE '(\cc d)$ under the map
$\Pi _A$ equals $\maclA _{\flat _1}(\cc d)$.
\end{enumerate}
\end{prop}

\par

%
%
%

\section{Paley-Wiener properties for
Bargmann-Pilipovi{\'c} spaces}\label{sec2}

\par

In this section we consider spaces of compactly supported functions with interiors in
$\maclA _s(\cc d)$ or in $\maclA _s'(\cc d)$. We show that the images of
such functions under the reproducing kernel $\Pi _A$ are equal to
$\maclA _s(\cc d)$, for some other choice of
$s\le \flat _1$. In the first part we state the main results given in Theorems
\ref{Thm:MainResult1}--\ref{Thm:MainResult2}. They are straight-forward
consequences of Propositions \ref{Prop:MainResult1}, where more detailed
information concerning involved constants are given. Thereafter we deduce
results which are needed for their proofs. Depending of the choice of $s$,
there are several different situations for characterizing $\maclA _s(\cc d)$.
This gives rise to a quite large flora of main results, where each one takes
care of one situation.

\par

In order to present the main results, we need the following definition.

%

\par

\begin{defn}\label{Def:RClass}
Let $t_1,t_2\in \rr d_+$ be such that $t_1\le t_2$. Then the set
$\maclR ^{*}_{t_1,t_2} (\cc d)$ consists of all
non-zero positive Borel measures
$\nu$ on $\cc d$ such that the following is true:
\begin{enumerate}
\item $d\nu (z_1,\dots ,z_d)$ is radial symmetric in each variable $z_j$;

\vrum

\item the support of $\nu$ contains
\begin{align*}
& \sets {z=(z_1,\dots ,z_d)\in \cc d}{|z_j|= t_{1,j} \ \text{for every}\  j=1,\dots ,d}
\intertext{and is contained in}
& \sets {z=(z_1,\dots ,z_d)\in \cc d}{|z_j|\le t_{2,j} \ \text{for every}\  j=1,\dots ,d}.
\end{align*}
\end{enumerate}
The set of compactly supported, positive, bounded and radial symmetric measures
is given by
$$
\maclR ^{*} (\cc d) \equiv \bigcup _{t_1\le t_2\in \rr d_+}
\maclR ^{*} _{t_1,t_2}(\cc d).
$$
\end{defn}

\par

\begin{rem}\label{Rem:RClass1}
Let $t_1,t_2\in \rr d_+$ be such that $t_1\le t_2$, $p\in [1,\infty ]$, 
$\maclR ^p_{t_1,t_2} (\cc d)$ be the set of all non-negative
$F\in L^p(\cc d)$ such that (1) and (2) in Definition \ref{Def:RClass}
holds with $F$ in place of $\nu$ and let
$$
\maclR ^p (\cc d) \equiv \bigcup _{t_1\le t_2\in \rr d_+}
\maclR ^p _{t_1,t_2}(\cc d).
$$
Then it is clear that $\maclR ^p _{t_1,t_2}(\cc d)$ and $\maclR ^p (\cc d)$ 
decrease with $p$ and are contained in $\maclR ^{*} _{t_1,t_2}(\cc d)$ and
$\maclR ^{*} (\cc d)$, respectively. In particular, these sets contain
$\maclR ^\infty _{t_1,t_2}(\cc d)$ and $\maclR ^\infty (\cc d)$, respectively,
in \cite{NaPfTo1}.
\end{rem}

\par

\begin{rem}\label{Rem:RClass2}
It is clear that the sets in Definition \ref{Def:RClass}
and Remark \ref{Rem:RClass1} are invariant
under multiplications with positive measurable, locally
bounded functions on $\cc d$ which are radial symmetric in each complex
variable $z_j$ in $z=(z_1,\dots ,z_d)\in \cc d$. In particular, they are
invariant under multiplications with $e^{t|z|^2}$ for every $t\in \mathbf R$.
\end{rem}

\par

\begin{rem}\label{Rem:RClass3}
Let $t_1,t_2\in \rr d_+$ be such that $t_1\le t_2$ and
$\nu \in \maclR ^{*} _{t_1,t_2}$. By Riesz representation theorem
it follows that
$$
d\nu (z) = d\theta d\nu _0(r),
$$
where
\begin{alignat*}{2}
z_j &=r_je^{i\theta _j}, &
\qquad
\theta
&=
(\theta _1,\dots ,\theta _d)\in [0,2\pi )^d,
\\[1ex]
r &=(r_1,\dots ,r_d)\in \overline {\mathbf R}_+^{\, d},&
\qquad
z
&=
(z_1,\dots ,z_d)\in \cc d
\end{alignat*}
and some non-zero positive Borel measures
$\nu _0$ on $\overline {\mathbf R}_+^{\, d}$ such that
the support of $\nu _0$ contains $t_1$ and is contained in
$$
\sets {r\in \overline {\mathbf R}_+^{\, d}}{r\le t_{2}}.
$$
\end{rem}

\par

\subsection{Main results}

\par

Our main investigations concern mapping properties of
operators of the form
\begin{equation}\label{Eq:MainMap}
F\mapsto \Pi _A(F\cdot \nu )
\end{equation}
when acting on the spaces given in Definition \ref{Def:PilPowerSpaces} (2).

\par

Before stating the main results we need the following
lemmas, which explain some properties of the map \eqref{Eq:MainMap}
when acting on the monomials $e_\alpha (z)$.

\par

\begin{lemma}\label{Lemma:HermiteGroch}
Let $t_1,t_2\in \rr d_+$ be such that $t_1\le t_2$, $\nu \in
\mathcal R^* _{t_1,t_2}(\cc d)$ and let $\nu _0$ be the same as in
Remark \ref{Rem:RClass3}. Then
\begin{align}
\Pi _A(e_\alpha \cdot  \nu ) &= \varsigma _\alpha \alpha !^{-1}e_\alpha ,
\qquad \alpha \in \nn d,
\label{Eq:BasicMap}
\intertext{where}
\varsigma _\alpha &= 2^d  \int  e^{-|r|^2}r^{2\alpha}\, d\nu _0(r)
\label{Eq:varsigmaDef}
\end{align}
satisfies
\begin{equation}\label{Eq:Estvarsigma}
t_1^{2\alpha}e^{-|t_2|^2}\lesssim \varsigma _\alpha \lesssim t_2^{2\alpha},
\qquad \alpha \in \nn d.
\end{equation}
\end{lemma}

\par

\begin{proof}
By using polar coordinates
in each complex variable when integrating we get
\begin{multline}\label{Eq:PiFalpha}
(\Pi _A(e_\alpha \cdot  \nu ))(z) = \pi ^{-d}
\alpha !^{-\frac 12}\int _{\cc d}w^\alpha e^{(z,w)-|w|^2}\, d\nu (w)
\\[1ex]
=
\pi ^{-d} \alpha !^{-\frac 12}
\int _{\Delta _{t_2}} I _\alpha (r,z)r^\alpha e^{-|r|^2}\, d\nu _0(r),
\end{multline}
where
\begin{align}
I _\alpha (r,z)
&=
\int _{[0,2\pi )^d}e^{i\scal \alpha \theta}
\left (
\prod _{j=1}^d
e^{z_jr_je^{-i\theta _j}}
\right )
\, d\theta
=
\prod _{j=1}^d I_{\alpha _j}(r_j,z_j)
\label{Eq:IalphaDef}
\intertext{with}
I_{\alpha _j}(r_j,z_j)
&=
\int _0^{2\pi}e^{i\alpha _j\theta _j}e^{z_jr_je^{-i\theta _j}}\, d\theta _j.
\notag
\end{align}
By Taylor expansions we get
\begin{multline*}
I_{\alpha _j}(r_j,z_j)
=
\int _0^{2\pi}e^{i\alpha _j\theta _j}
\left (
\sum _{k=0}^\infty
\frac {z_j^kr_j^ke^{-ik\theta _j}}{k!}
\right )
\, d\theta _j
\\[1ex]
=
\sum _{k=0}^\infty
\left (
\left (
\int _0^{2\pi}e^{i(\alpha _j-k)\theta _j}
\, d\theta _j 
\right )
\frac {z_j^kr_j^k}{k!}
\right )
=
\frac {2\pi z_j^{\alpha _j}r_j^{\alpha _j}}{\alpha _j!},
\end{multline*}
where the second equality is justified by
$$
\sum _{k=0}^\infty
\left (
\int _0^{2\pi}|e^{i(\alpha _j-k)\theta _j}|
\, d\theta _j 
\left |
\frac {z_j^kr_j^k}{k!}
\right |
\right ) = e^{2\pi (|z_1r_1| + \cdots + |z_dr_d|)} <\infty 
$$
and Weierstrass' theorem.

\par

By inserting this into \eqref{Eq:PiFalpha} and \eqref{Eq:IalphaDef} we get
\begin{multline*}
(\Pi _A(e_\alpha \cdot  \nu ))(z) = \pi ^{-d}\alpha !^{-\frac 12}
\int _{\Delta _{t_2}} 
(2\pi )^d\frac {r^\alpha z^\alpha}{\alpha !}
r^\alpha e^{-|r|^2}\, d\nu _0(r)
\\[1ex]
=
\varsigma _\alpha \alpha !^{-1}e_\alpha (z),
\end{multline*}
and \eqref{Eq:BasicMap} follows.

\par

The estimates in \eqref{Eq:Estvarsigma} are straight-forward consequences
of \eqref{Eq:varsigmaDef} and the support properties of $\nu _0$. The details
are left for the reader.
\end{proof}

\par

By replacing $\nu$ in the previous lemma with suitable
radial symmetric compactly supported distributions we get the following.

\par

\begin{lemma}\label{Lemma:HermiteGroch2}
Let $s>1$, $t_1,t_2\in \rr d_+$ be such that $t_1\le t_2$, $\nu (z)\in
\maclE _s'(\cc d)$ be radial symmetric in each $z_j$
such that
$$
\supp \nu
\subseteq 
\sets {z\in \cc d}{ t_{1,j}\le |z_j|\le t_{2,j}}
$$
Then \eqref{Eq:BasicMap} holds with
\begin{equation}\label{Eq:varsigmaDef2}
\varsigma _\alpha = 2^d  \alpha !^{-1}\scal {\nu _0}{\phi _\alpha},\quad
\phi _\alpha (r) = e^{-|r|^2}r^{2\alpha}r_1\cdots r_d,\ r\in \rr d_+,\ \alpha \in \nn d
\end{equation}
for some $\nu _0\in \maclE _s'(\rr d_+)$ with
$$
\supp \nu _0\subseteq \sets {r\in \rr d_+}{t_{1,j}\le r_j\le t_{2,j}\ \text{for every}\ j}.
$$
Furthermore,
\begin{equation}\label{Eq:Estvarsigma2}
\varsigma _\alpha \lesssim t_2^{2\alpha}.
\end{equation}
\end{lemma}

\par

\begin{proof}
By using polar coordinates in each complex variable, the pull-back formula
\cite[Theorem 6.1.2]{Ho1} and Fubbini's
theorem for distributions and ultra-distributions, we get
\begin{multline}\label{Eq:PiFalpha2}
(\Pi _A(e_\alpha \cdot  \nu ))(z) = \pi ^{-d}\scal \nu{e^{(z,\cdo )-|\cdo |^2}e_\alpha}
\\[1ex]
=\pi ^{-d}\alpha !^{-\frac 12}\scal {\nu _0\otimes 1_{[0,2\pi )^d}}\Psi ,
\end{multline}
for some $\nu _0\in \maclE _s'(\rr d_+)$, where
$$
\Psi (r,\theta ) = e^{-|r|^2}\prod _{j=1}^d
\left (
e^{z_jr_je^{-i\theta _j}}r_j^{\alpha _j}e^{i\alpha _j\theta _j}r_j
\right ).
$$
By the same arguments as in the proof of Lemma \ref{Lemma:HermiteGroch}
we get
$$
\scal {1_{[0,2\pi )^d}}{\Psi (r,\cdo )} = (2\pi )^d \alpha !^{-1}
e^{-|r|^2}r^{2\alpha}r_1\cdots r_d\cdot z^\alpha 
=
(2\pi )^d \alpha !^{-\frac 12}\phi _\alpha (r)e_\alpha (z),
$$
and \eqref{Eq:BasicMap} follows with $\varsigma _\alpha$ given by
\eqref{Eq:varsigmaDef2}, by combining the latter identity with
\eqref{Eq:PiFalpha2}.

\par

The support assertions for $\nu _0$ follow from the support properties of $\nu$,
and the estimate \eqref{Eq:Estvarsigma2} follows from the fact that
$t_2^{-2\alpha}\phi _\alpha$ is a bounded set in $\maclE _s(\rr d_+)$ with respect
to $\alpha$ in the support of $\nu _0$. This gives the result.
\end{proof}

\par

Due to Lemma \ref{Lemma:HermiteGroch2} we let 
$\maclE _{RS}'(\cc d)$ be the set of all
$$
\nu (z)\in \bigcup _{s>1}\maclE _s'(\cc d)
$$
which are radial symmetric in each $z_j$ and such that
$0\notin \supp (\nu )$.

\par

\begin{thm}\label{Thm:MainResult1}
Let $s_1\in (0,\frac 12)$, $s_2\in [0,\frac 12)$ and
$\nu \in \maclR ^*(\cc d)$ ($\nu \in \maclE _{RS}'(\cc d)$).
Then the map \eqref{Eq:MainMap} from $\maclA _0(\cc d)$ to
$A(\cc d)$ is uniquely extendable to homeomorphisms (continuous
mappings) on
$$
\maclA _{s_2}(\cc d),
\quad
\maclA _{0,s_1}(\cc d),
\quad
\maclA _{0,s_1}'(\cc d)
\quad \text{and on}\quad
\maclA _{s_2}'(\cc d).
$$
\end{thm}

\par

\begin{thm}\label{Thm:MainResult2}
Let $\sigma ,\sigma _0\in \mathbf R_+$ and $\nu \in \maclR ^*(\cc d)$
($\nu \in \maclE _{RS}'(\cc d)$). Then the following is true:
\begin{enumerate}
\item If $\sigma _0=\frac \sigma {2\sigma +1}$, then the map
\eqref{Eq:MainMap} from $\maclA _0(\cc d)$ to
$A(\cc d)$ is uniquely extendable to homeomorphisms (continuous
mappings) from $\maclA _{\flat _{\sigma _0}}'(\cc d)$ to
$\maclA _{\flat _{\sigma}}'(\cc d)$, and from
$\maclA _{0,\flat _{\sigma _0}}'(\cc d)$ to
$\maclA _{0,\flat _{\sigma}}'(\cc d)$;

\vrum

\item if $\sigma >\frac 12$ and $\sigma _0=\frac {\sigma}{2\sigma -1}$, then
the map \eqref{Eq:MainMap} from $\maclA _0(\cc d)$ to
$A(\cc d)$ is uniquely extendable to homeomorphisms (continuous
mappings) from $\maclA _{0,\flat _{\sigma _0}}'(\cc d)$ to
$\maclA _{\flat _{\sigma}}(\cc d)$, and from
$\maclA _{\flat _{\sigma _0}}'(\cc d)$ to
$\maclA _{0,\flat _{\sigma}}(\cc d)$;

\vrum

\item if $\sigma < \frac 12$ and $\sigma _0=\frac {\sigma}{1-2\sigma}$, then
the map \eqref{Eq:MainMap} from $\maclA _0(\cc d)$ to
$A(\cc d)$ is uniquely extendable to homeomorphisms (continuous
mappings) from $\maclA _{\flat _{\sigma _0}}(\cc d)$ to
$\maclA _{\flat _{\sigma}}(\cc d)$, and from
$\maclA _{0,\flat _{\sigma _0}}(\cc d)$ to
$\maclA _{0,\flat _{\sigma}}(\cc d)$.
\end{enumerate}
\end{thm}

\par

The limit cases for the situations in the previous theorem are treated in the
next result.

\par

\begin{thm}\label{Thm:MainResult3}
Let $\nu \in \maclR ^*(\cc d)$ ($\nu \in \maclE _{RS}'(\cc d)$) and
$s=\sigma =\frac 12$.
Then the following is true:
\begin{enumerate}
\item The map \eqref{Eq:MainMap} from $\maclA _0(\cc d)$ to
$A(\cc d)$ is uniquely extendable to homeomorphisms (continuous
mappings) from $\maclA _{0,\flat _\sigma}'(\cc d)$ to
$\maclA _{0,s}'(\cc d)$, and from $\maclA _{\flat _\sigma}'(\cc d)$ to
$\maclA _{0,s}(\cc d)$;

\vrum

\item The map \eqref{Eq:MainMap} from $\maclA _0(\cc d)$ to
$A(\cc d)$ is uniquely extendable to homeomorphisms (continuous
mappings) from $\maclA _{0,s}'(\cc d)$ to
$\maclA _{\flat _\sigma}(\cc d)$, and from
$\maclA _{0,s}(\cc d)$ to
$\maclA _{0,\flat _\sigma}(\cc d)$.
\end{enumerate}
\end{thm}

\par

\begin{rem}
Since
$$
\mascE '(\cc d)\cap L^\infty (\cc d) \subseteq \mascE '(\cc d)
\subseteq \maclE _s'(\cc d),
$$
Theorem \ref{Thm:MainResult1} still holds true after $\maclE _s'$
has been replaces by $\mascE '$ in (6).
\end{rem}

\par


\par

The following result is an essential part of the proof of
Theorem \ref{Thm:MainResult1}.

\par

\begin{prop}\label{Prop:MainResult1}
Let $r_0\in \mathbf R_+$,
$s\in [0,\frac 12)$ and $\nu \in \maclR ^*(\cc d)$ be
fixed, and let
$\varsigma _\alpha $ be as in \eqref{Eq:varsigmaDef}.
Then the following is true:
\begin{enumerate}
\item The map $F\mapsto \Pi _A(F\cdot \nu )$ from $\maclA _0(\cc d)$
to $A(\cc d)$ is uniquely extendable to a homeomorphism on
$\maclA _0'(\cc d)$, and
\begin{equation}\label{Eq:OpCoeffRel}
c(\Pi _A(F\cdot \nu ),\alpha ) = \varsigma _\alpha \alpha !^{-1}
c(F,\alpha ),
\qquad F\in \maclA _0'(\cc d),\  \alpha \in \nn d\text ;
\end{equation}

\vrum

\item it holds
\begin{align}
|c(F,\alpha )|
&\lesssim
e^{-r|\alpha |^{\frac 1{2s}}}
\label{Eq:EstCoeffSmalls1}
\intertext{for some $r\in \mathbf R_+$ such that
$r<r_0$, if and only if}
|c(\Pi _A(F\cdot \nu ),\alpha )|
&\lesssim
e^{-r|\alpha |^{\frac 1{2s}}}
\label{Eq:EstCoeffSmalls2}
\end{align}
for some $r\in \mathbf R_+$ such that $r<r_0$;

\vrum

\item it holds
\begin{align}
|c(F,\alpha )|
&\lesssim
e^{r|\alpha |^{\frac 1{2s}}}
\label{Eq:EstCoeffSmalls1B}
\intertext{for some $r\in \mathbf R_+$ such that
$r<r_0$, if and only if}
|c(\Pi _A(F\cdot \nu ),\alpha )|
&\lesssim
e^{r|\alpha |^{\frac 1{2s}}}
\label{Eq:EstCoeffSmalls2B}
\end{align}
for some $r\in \mathbf R_+$ such that $r<r_0$.
\end{enumerate}
\end{prop}

\par

Here it is understood that the signs in the exponents in
\eqref{Eq:EstCoeffSmalls1} and \eqref{Eq:EstCoeffSmalls2}
agree.

\par

\begin{proof}
The assertion (1) is an immediate consequence of
\eqref{Eq:BasicMap} in Lemma \ref{Lemma:HermiteGroch}.
In fact, by Lemma \ref{Lemma:HermiteGroch} and \eqref{Eq:BasicMap},
the only possible extension of $F\mapsto \Pi _A(F\cdot \nu )$ is to let
\begin{align}
\Pi _A(F\cdot \nu)(z) &= \sum _{\alpha \in \nn d}
\left (
\varsigma _\alpha \alpha !^{-1}c(F,\alpha )
\right )
e_\alpha (z)
\label{Eq:ExtPiA}
\intertext{when}
F(z) &= \sum _{\alpha \in \nn d}c(F,\alpha )e_\alpha (z),
\label{Eq:FExp}
\end{align}
which obviously defines a continuous map on
$\maclA _0'(\cc d)$.

\par

Since
$$
t^\alpha \alpha ! \lesssim e^{r|\alpha |^{\frac 1{2s}}}
\quad \text{and}\quad
t_1^{2\alpha}\lesssim \varsigma _\alpha \lesssim t_2^{2\alpha},
$$
for every $r\in \mathbf R_+$, it follows from \eqref{Eq:OpCoeffRel} that
$$
e^{-r|\alpha |^{\frac 1{2s}}}
\lesssim
\frac {|c(\Pi _A(F\cdot \nu ),\alpha )|}{|c(F,\alpha )|}
\lesssim
e^{r|\alpha |^{\frac 1{2s}}}
$$
for every for every $r\in \mathbf R_+$. This gives (2) and (3).
\end{proof}

\par

In order to prove Theorem \ref{Thm:MainResult2} we need
the next proposition. The result is a straight-forward
consequence of \eqref{Eq:BasicMap}, \eqref{Eq:Estvarsigma}
in Lemma \ref{Lemma:HermiteGroch}
and Proposition \ref{Prop:MainResult1} (1). The details are
left for the reader.

\par

\begin{prop}\label{Prop:MainResult2A}
Let $h\in \mathbf R_+$, $\tau \in \mathbf R$,
$t_1,t_2\in \rr d_+$,
$\nu \in \maclR ^*_{t_1,t_2}(\cc d)$ and $\varsigma _\alpha $
be as in \eqref{Eq:varsigmaDef}.
Then
\begin{align}
|c(F,\alpha )|
&\lesssim
h^{|\alpha |}\alpha !^{\tau}\notag
\\[1ex]
&\Rightarrow \quad
|c(\Pi _A(F\cdot \nu ),\alpha )|
\lesssim
h^{|\alpha |}t_2^{2\alpha }\alpha !^{\tau -1}
\label{Eq:EstCoeffSigmaBig1}
\intertext{and}
|c(\Pi _A(F\cdot \nu ),\alpha )|
&\lesssim
h^{|\alpha |}t_1^{2\alpha }\alpha !^{\tau -1}
\notag
\\[1ex]
&\Rightarrow \quad
|c(F,\alpha )|
\lesssim
h^{|\alpha |}\alpha !^{\tau}
\label{Eq:EstCoeffSigmaBig2}
\end{align}
\end{prop}

\par

\begin{proof}[Proof of Theorems \ref{Thm:MainResult1}
and \ref{Thm:MainResult2}]
Theorem \ref{Thm:MainResult1} is a straight-forward consequence
of Proposition \ref{Prop:MainResult1}. The details are left
for the reader.

\par

By letting $\sigma >0$, $\sigma _0=\frac \sigma{2\sigma +1}$
and $\tau = \frac 1{2\sigma _0}$, then $\tau -1
=\frac 1{2\sigma}$. Hence \eqref{Eq:EstCoeffSigmaBig1}
and \eqref{Eq:EstCoeffSigmaBig2} give
\begin{equation}\label{Eq:MainResult2Case1}
\begin{alignedat}{2}
|c(F,\alpha )|
&\lesssim
h^{|\alpha |}\alpha !^{\frac 1{2\sigma _0}}&
\quad &\text{for some (for every)}\ h>0
\\[1ex]
&&&\Leftrightarrow
\\[1ex]
|c(\Pi _A(F\cdot \nu ),\alpha )|
&\lesssim
h^{|\alpha |}\alpha !^{\frac 1{2\sigma}}&
\quad &\text{for some (for every)}\ h>0.
\end{alignedat}
\end{equation}
Theorem \ref{Thm:MainResult2} (1) now follows from 
Proposition \ref{Prop:MainResult1} (1) and
\eqref{Eq:MainResult2Case1}.

\par

If instead $\sigma >\frac 12$,
$\sigma _0=\frac \sigma{2\sigma -1}$
and $\tau = \frac 1{2\sigma _0}$, then $\tau -1
=-\frac 1{2\sigma}$. Hence \eqref{Eq:EstCoeffSigmaBig1}
and \eqref{Eq:EstCoeffSigmaBig2} give
\begin{equation}\label{Eq:MainResult2Case2}
\begin{alignedat}{2}
|c(F,\alpha )|
&\lesssim
h^{|\alpha |}\alpha !^{\frac 1{2\sigma _0}}&
\quad &\text{for some (for every)}\ h>0
\\[1ex]
&&&\Leftrightarrow
\\[1ex]
|c(\Pi _A(F\cdot \nu ),\alpha )|
&\lesssim
h^\alpha !^{-\frac 1{2\sigma}}&
\quad &\text{for some (for every)}\ h>0.
\end{alignedat}
\end{equation}
Theorem \ref{Thm:MainResult2} (2) now follows from 
Proposition \ref{Prop:MainResult1} (1) and
\eqref{Eq:MainResult2Case2}.

\par

If instead $\sigma <\frac 12$,
$\sigma _0=\frac \sigma{1-2\sigma}$
and $\tau = -\frac 1{2\sigma _0}$, then $\tau -1
=-\frac 1{2\sigma}$. Hence \eqref{Eq:EstCoeffSigmaBig1}
and \eqref{Eq:EstCoeffSigmaBig2} give
\begin{equation}\label{Eq:MainResult2Case3}
\begin{alignedat}{2}
|c(F,\alpha )|
&\lesssim
h^{|\alpha |}\alpha !^{-\frac 1{2\sigma _0}}&
\quad &\text{for some (for every)}\ h>0
\\[1ex]
&&&\Leftrightarrow 
\\[1ex]
|c(\Pi _A(F\cdot \nu ),\alpha )|
&\lesssim
h^\alpha !^{-\frac 1{2\sigma}}&
\quad &\text{for some (for every)}\ h>0.
\end{alignedat}
\end{equation}
Theorem \ref{Thm:MainResult2} (3) now follows from 
Proposition \ref{Prop:MainResult1} (1) and
\eqref{Eq:MainResult2Case3}.

\par

Finally, Theorem \ref{Thm:MainResult3} follows
by similar arguments, letting $\tau =1$ when
proving (1) and letting $\tau =0$ when proving
(2) in Theorem \ref{Thm:MainResult3}. The details are
left for the reader.
\end{proof}

\par

\section{Characterizations of Pilipovi{\'c}
spaces}\label{sec3}

\par

In this section we combine Lemma \ref{GrochSpaceBargm} with
Theorems \ref{Thm:MainResult1}--\ref{Thm:MainResult3} to
get characterizations of Pilipovi{\'c} spaces.

\par

We shall perform such characterizations by considering mapping properties
of extensions of the map
$$
F\mapsto \Theta _F =\Theta _{F,r}
$$
from $\maclA _0 (\cc {d})$ to $C^\infty (\rr d)$, where 
\begin{equation}\label{Eq:FMapDef}
\Theta _{F,r}(x)\equiv \iint _{D_r(0)} F(y+i\eta )
e^{-(\frac 12|x-y|^2+|y|^2+|\eta |^2)}
e^{i(\frac 12\scal y\eta -\scal x\eta )}\, dyd\eta .
\end{equation}
Here we have identified $D_r(0)\subseteq \cc d$ with the polydisc
$$
\sets {(x,\xi )\in \rr {2d}}{x_j^2+\xi _j^2<r_j^2,\ j=1,\dots ,d}
$$
in $\rr {2d}$ when $r=(r_1,\dots ,r_d)\in \rr d_+$.
We notice that $\Theta _F$ equals $f$ in \eqref{Eq:fSTFTAdjDef},
if in addition $F\in L^\infty (\rr {2d})$.

\par

We recall that the Bargmann transform is homeomorphic between the spaces in
\eqref{Eq:PilSpaces} and \eqref{Eq:PowerSpaces} when $s_1\in \mathbf R_\flat$
and $s_2\in \overline {\mathbf R}_\flat$. The following results of Paley-Wiener types
for Pilipovi{\'c} spaces, follow from these facts and by some straight-forward
combinations of Lemma \ref{GrochSpaceBargm} and Theorems
\ref{Thm:MainResult1}--\ref{Thm:MainResult3}. The details are left for the reader.

\par

\begin{thm}\label{Thm:MainResult1AH}
Let $s_1\in [0,\frac 12)$, $s_2\in (0,\frac 12)$ and $r\in \rr d_+$.
Then the map \eqref{Eq:FMapDef} from $\maclA _0(\cc d)$ to
$C^\infty (\rr d)$ is uniquely extendable to homeomorphisms (continuous
mappings) from the spaces in \eqref{Eq:PilSpaces} to corresponding
spaces in \eqref{Eq:PowerSpaces}.
\end{thm}

\par

\begin{thm}\label{Thm:MainResult2AH}
Let $\sigma ,\sigma _0\in \mathbf R_+$ and $r\in \rr d_+$. Then the following is true:
\begin{enumerate}
\item If $\sigma _0=\frac \sigma {2\sigma +1}$, then the map
\eqref{Eq:FMapDef} from $\maclA _0(\cc d)$ to
$C^\infty (\rr d)$ is uniquely extendable to homeomorphisms
from $\maclA _{\flat _{\sigma _0}}'(\cc d)$ to
$\maclH _{\flat _{\sigma}}'(\rr d)$, and from
$\maclA _{0,\flat _{\sigma _0}}'(\cc d)$ to
$\maclH _{0,\flat _{\sigma}}'(\rr d)$;

\vrum

\item if $\sigma >\frac 12$ and $\sigma _0=\frac {\sigma}{2\sigma -1}$, then
the map \eqref{Eq:FMapDef} from $\maclA _0(\cc d)$ to
$C^\infty (\rr d)$ is uniquely extendable to homeomorphisms
from $\maclA _{0,\flat _{\sigma _0}}'(\cc d)$ to
$\maclH _{\flat _{\sigma}}(\rr d)$, and from
$\maclA _{\flat _{\sigma _0}}'(\cc d)$ to
$\maclH _{0,\flat _{\sigma}}(\rr d)$;

\vrum

\item if $\sigma < \frac 12$ and $\sigma _0=\frac {\sigma}{1-2\sigma}$, then
the map \eqref{Eq:FMapDef} from $\maclA _0(\cc d)$ to
$C^\infty (\rr d)$ is uniquely extendable to homeomorphisms
from $\maclA _{\flat _{\sigma _0}}(\cc d)$ to
$\maclH _{\flat _{\sigma}}(\rr d)$, and from
$\maclA _{0,\flat _{\sigma _0}}(\cc d)$ to
$\maclH _{0,\flat _{\sigma}}(\rr d)$.
\end{enumerate}
\end{thm}

\par

\begin{thm}\label{Thm:MainResult3AH}
Let $s=\sigma =\frac 12$ and $r\in \rr d_+$.
Then the following is true:
\begin{enumerate}
\item The map \eqref{Eq:FMapDef} from $\maclA _0(\cc d)$ to
$C^\infty (\rr d)$ is uniquely extendable to homeomorphisms (continuous
mappings) from $\maclA _{0,\flat _\sigma}'(\cc d)$ to
$\maclH _{0,s}'(\rr d)$, and from $\maclA _{\flat _\sigma}'(\cc d)$ to
$\maclA _{0,s}(\rr d)$;

\vrum

\item The map \eqref{Eq:FMapDef} from $\maclA _0(\cc d)$ to
$C^\infty (\rr d)$ is uniquely extendable to homeomorphisms (continuous
mappings) from $\maclA _{0,s}'(\cc d)$ to
$\maclH _{\flat _\sigma}(\rr d)$, and from
$\maclA _{0,s}(\cc d)$ to
$\maclH _{0,\flat _\sigma}(\rr d)$.
\end{enumerate}
\end{thm}

\par

\par


\end{document}